\documentclass[twoside, 11pt, leqno]{amsart}
 \usepackage{amsmath, amssymb, amsfonts, amsthm, amscd}
 \usepackage{mathrsfs}
 \usepackage{cancel}
\usepackage{kpfonts}
 \usepackage[margin=1.5in]{geometry}

\theoremstyle{plain}
\numberwithin{equation}{section}

\newtheorem{theorem}{Theorem}[section]

\newtheorem{lemma}[theorem]{Lemma}
\newtheorem{corollary}[theorem]{Corollary}

\newtheorem{definition}[theorem]{Definition}

\newcommand{\fracsm}[2]{\begin{matrix}\frac{#1}{#2}\end{matrix}}

\newcommand{\beq}{\begin{equation}}
\newcommand{\eeq}{\end{equation}}

\newcommand{\Reals}{\mathbb{R}}

\DeclareMathOperator{\diverg}{div}

\thanks{The author was partially supported by the NSF research grant DMS-1311795.}
\begin{document}

\author{Niels Martin M\o{}ller}
\address{Niels Martin M\o{}ller, Department of Mathematics, Princeton University, NJ.}
\email{moller@math.princeton.edu}

\title{Non-existence for self-translating solitons}


\begin{abstract}
This paper establishes geometric obstructions to the existence of complete, properly embedded, mean curvature flow self-translating solitons $\Sigma^n\subseteq \Reals^{n+1}$, generalizing previously known non-existence conditions such as cylindrical boundedness.
\end{abstract}

\today

\maketitle

\section{introduction}

Let $\Sigma^n\subseteq \Reals^{n+1}$ be a smooth, connected, self-translating $n$-dimensional hypersurface, with translation direction $\mathbf{a}\in\Reals^{n+1}$, which means an oriented surface satisfying
\beq\label{TransEq}
H_{\Sigma}=\langle \mathbf{a},\nu_{\Sigma}\rangle,
\eeq
where $H_{\Sigma}$ is the mean curvature $H_{\Sigma}=\sum_{i=1}^n\kappa_i$ (sum of principal curvatures) with respect to the unit normal $\nu_{\Sigma}$.

Such surfaces appear naturally in the singularity theory of mean curvature flow (associated with the so-called Type II singularities), the simplest example being in the case of (immersed) planar curves ($n=1$), where Calabi's grim reaper self-translating soliton curve $y=\log(\cos x)$ appears as a parabolic rescaling limit (see \cite{Ang}, and see f.ex. \cite{Sh13} for an exposition with references to the higher-dimensional cases). Another related viewpoint is that the mean curvature flow equation, being a nonlinear geometric heat equation, can be expected to admit solitary solutions, i.e. solitons, where the nonlinearity balances out the diffusive nature of the heat dissipation - and most importantly those solutions which move along a (conformal) Killing field in $\Reals^{n+1}$. The self-translating hypersurfaces are in this context those which move, as a mean curvature flow $\{\Sigma_t\}$, along a pure translational field:
\[
\Sigma_t = \Sigma_0 + t\:\mathbf{a},
\]
where $\mathbf{a}\in\Reals^{n+1}$ is a fixed (non-zero) vector. We will in the below generally normalize the translation direction vector to $\mathbf{a}=e_{n+1}$, so that also the speed is fixed $|\mathbf{a}|=1$.

Similarly to in the theory of classical minimal surfaces ($H=0$) in $\Reals^{n+1}$, the set of complete self-translating solitons that can exist is restricted by geometric constraints. Combining knowledge of explicit examples with a maximum principle, as well as by exploiting the symmetries of the equation, one may get such obstructions. Note in this connection that while the minimal surface equation has all rigid motions and homotheties as invariances, the self-translater equation (\ref{TransEq}) is a priori invariant (for fixed $\mathbf{a}$) only under motions that preserve both the translation direction and speed (i.e. that preserve $\mathbf{a}$ as a, not based, vector). In order to rigorously invoke the maximum principle, one is naturally led to consider asymptotical conditions. We consider here primarily one such (not purely asymptotical) condition which involves solid double paraboloidal "funnels" directed along the translation axis. The funnel condition is adapted so that it appears to be the sharp condition (with respect to growth rates, and the order of the height of the enclosed region) which allows one to use a "last point of touching" maximum principle argument, based on the Clutterbuck-Schn\"u{}rer-Schulze \cite{CSS07} (and Altschuler-Wu \cite{AW94}) solitons with an $O(n)$-symmetry in $\Reals^{n+1}$, the topology of these being $\mathbb{S}^n\times \Reals$ (resp. $(\mathbb{S}^n\times \Reals_+) \cup \{*\}$). This family is ample for reaching interesting conclusions on the set of possible translaters.

We denote by $\mathcal{W}_{R}$ the unique two-ended "winglike" self-translating hypersurface from \cite{CSS07} (the $R$-value coincides with the notation in the reference), with initial conditions given by the requirement that the tangent line of the generating curve at $(R,0\ldots,0)\in\Reals^{n+1}$ is parallel to the $e_{n+1}$-direction. By convention, $\mathcal{W}_{R}$ is rotationally symmetric with respect to the principal axis $\{(0,0,\ldots,0,x_{n+1})\}$.

\begin{definition}\label{def_funnel}
The ($e_{n+1}$-directed) solid parabolic funnel of aperture $R_0>0$, logarithmic parameter $\lambda>1$ and center locus at $y_0\in\Reals^{n+1}$, is the closed set $\mathcal{F}^\lambda_{R_0,y_0}\subseteq \Reals^{n+1}$ given by:
\[
\mathcal{F}^\lambda_{R_0,y_0}=\left\{(x_1,x_2,\ldots,x_{n+1})\in\Reals^{n+1}:f_{-}(|p|)\leq x_{n+1}\leq f_{+}(|p|),\: |p|\geq R_0\right\} + y_0,
\]
where $|p|=\sqrt{x_1^2+\ldots + x_n^2}$, and
\begin{align*}
&f_+(r)=\frac{r^2}{2(n-1)} + 1,\\
&f_-(r)=\frac{(r-R_0-2)^2}{2(n-1)}-\lambda\frac{\log\left(1 + (r-R_0-2)^2\right)}{2}-\frac{\pi(R_0+\fracsm{\pi}{2})+4}{2(n-1)}.
\end{align*}
\end{definition}

With the funnel condition, we have the following geometric obstruction:

\begin{theorem}\label{NoFunnels}
Let $\Sigma^n\subseteq\Reals^{n+1}$ ($n\geq 2)$ be a smooth, complete, properly embedded, $e_{n+1}$-self-translating hypersurface,
\beq\label{TransEq}
H_{\Sigma}=\langle e_{n+1},\nu_\Sigma\rangle.
\eeq
Then $\Reals^{n+1}\setminus\Sigma^n$ cannot contain any $e_{n+1}$-directed solid parabolic funnel.
\end{theorem}

We recall that in \cite{CSS07} it was shown that for the two ends (upper and lower branches, $\mathcal{W}^+_{R}$ and $\mathcal{W}^-_{R}$, respectively) of the winglike hypersurface $\mathcal{W}_{R}$, the graphs $(p,u^{\pm}_R(p))\in\Reals^{n+1}$ representing them (asymptotically in the exterior of large balls) satisfy:

\beq\label{CSS_asympt}
u^\pm_R(p)=\frac{|p|^2}{2(n-1)}-\ln |p|+C_{\pm}(R)+O(|p|^{-1}),\quad p\in\Reals^n,
\eeq
where the asymptotics refer to $|p|\to\infty$, and where the constants $C_{\pm}(R)$ are not explicitly known or determined. Definition \ref{def_funnel} is thus best thought of as an effective formulation of these asymptotics (and see Lemma \ref{in_funnel} below, which relates the definition to (\ref{CSS_asympt})), with the $R$-dependence of the separation of the ends furthermore made explicit. The two main reasons the funnel condition allows the analysis we have in mind are then: (1) Each funnel contains the respective winglike solution (Lemma \ref{in_funnel}), and (2) The family of funnels (in $R$, for fixed $y_0$ and $\lambda$) satisfies a compactness property for the additional set traced out as the aperture decreases in size (Lemma \ref{compactness}).

As an immediate corollary of Theorem \ref{NoFunnels}, we get the following:

\begin{corollary}\label{NoCyls}
Let $\Sigma^n\subseteq\Reals^{n+1}$ ($n\geq 2)$ be a smooth, complete, properly embedded, $e_{n+1}$-self-translating hypersurface,
\beq
H_{\Sigma}=\langle e_{n+1},\nu_\Sigma\rangle.
\eeq
Then $\Sigma$ cannot have an $e_{n+1}$-directed end contained in a generalized half-cylinder, while the $-e_{n-1}$-directed end is contained below the corresponding horizontal hyperplane. That is, $\Sigma^n$ cannot be contained in any subset of the form (for $h_0\in\Reals)$
\[
S_h=\{x_{n+1}\leq h_0\}\cup\Big(\Omega^n\times(h_0,\infty)\Big),
\]
where $\Omega^n\subseteq\Reals^n$ is a bounded open domain in the $x_1x_2\ldots x_n$-hyperplane.
\end{corollary}

Weakening the assumptions further, one obtains:
\begin{corollary}\label{NoCyls}
Let $\Sigma^n\subseteq\Reals^{n+1}$ ($n\geq 2)$ be a smooth, complete, properly embedded, $e_{n+1}$-self-translating hypersurface,
\beq
H_{\Sigma}=\langle e_{n+1},\nu_\Sigma\rangle.
\eeq
Then $\Sigma$ cannot be contained in a generalized $e_{n+1}$-directed cylinder $\Omega^n\times\Reals$, where $\Omega^n\subseteq\Reals^n$ is a bounded open domain in the $x_1x_2\ldots x_n$-hyperplane.
\end{corollary}

Finally, one concludes also the graph case considered in \cite{Sh13}, using there the method of limiting minimal hypersurfaces:

\begin{corollary}[Theorem 1.2.7 in \cite{Sh13}]
For $n\geq 2$, there is no complete translating graph $\Sigma \subseteq \Reals^{n+1}$ over a smooth bounded connected domain $\Omega\subseteq\Reals^n$.
\end{corollary}

It should be noted that there do exist many examples of complete, properly embedded self-translaters which do not satisfy the funnel condition: Every vertical hyperplane containing the vector $e_{n+1}$ constitutes an example in $\Reals^{n+1}$, as do in $\Reals^3$ the desingularized unions of vertical $2$-planes and grim reaper cylinders $\Sigma^2=\Gamma\times\Reals$ in the work of X.H. Nguyen (see \cite{Ng09}-\cite{Ng12}).

I thank Hojoo Lee, Hyunsuk Kang and Jaigyoung Choe for their hospitality at the Korean Institute for Advanced Study, in August 2013, where most of this work was carried out, and particularly to Hojoo Lee for discussions where he pointed out the obstruction question in reference to L. Shahriyari's work. After this work was completed, the author learned that some special cases (e.g. the cylindrically bounded, not necessarily graphical case, which in the present paper is Corollary \ref{NoCyls}), were already independently described in the interesting recent paper \cite{MSS14}.

\section{Estimates on the winglike solutions}
In this section, we prove the following, where $\mathcal{W}_R$ refers to the Clutterbuck-Schn\"u{}rer-Schulze two-ended winglike solution which intersects the hyperplane $\{x_{n+1}=0\}$ perpendicularly at the sphere $x_1^2+x_2^2+\ldots x_n^2 = R^2$, i.e. the union of the upper respectively lower branches $W^{+}_R$ and $W^{-}_R$ constructed in Lemma 2.3 in \cite{CSS07}.

\begin{lemma}\label{in_funnel}
The winglike solution $\mathcal{W}_R$ is contained in the solid double funnel $\mathcal{F}^\lambda_R=\mathcal{F}^\lambda_{R,0}$ ($R\geq 0$ and $\lambda \in [0,1]$).
\end{lemma}

\subsection{Global bounds via monotonicity identities for the ODE}
For the required bounds, one can start with a simple estimation method which has also proven useful previously (to the author and S. Kleene in \cite{KM}, for the self-shrinking soliton equation). It is on one hand the standard idea, to obtain bounds for the solutions of an elliptic nonlinear equation, of "freezing" the nonlinearities in a suitable way, but in the current context of solitons (and as only became clear to us later, when seeing the treatment in \cite{GGS10}, p. 123) such a transformation is best understood as a form of (elliptic, rotationally symmetric) cousin of the (parabolic, general) monotonicity identity first found by Huisken, where the Gaussian density or heat kernel naturally appears from solving the linearized "skeleton" of the (elliptic or parabolic) equation.

Consider the equation for the derivative $\varphi(r)=V'(r)$ of graphs which generate by rotation a self-translater (see \cite{CSS07} and \cite{AW94}):
\beq\label{phiODE}
\varphi'=(1+\varphi^2)\left(1-\fracsm{n-1}{r}\varphi\right).
\eeq

Inspired by the expected asymptotics (i.e. Definition \ref{def_funnel}, or Equation (\ref{CSS_asympt})), we rewrite as
\beq
\varphi'=\frac{1+\varphi^2}{r^2}\left(r^2-(n-1)r\varphi\right),
\eeq
and freezing the nonlinearity $g(r)=\frac{1+(\varphi(r))^2}{r^2}$, solve the linear equation
\beq
\phi' + (n-1)rg\varphi = r^2g,
\eeq
for $\varphi$. This gives the integral (monotonicity) identity
\beq\label{intident}
\begin{split}
\varphi(r) &= e^{-(n-1)\int_{r_0}^r\fracsm{1+(\varphi(s))^2}{s}ds}\left[\int_{r_0}^r e^{(n-1)\int_{r_0}^t\fracsm{1+(\varphi(s))^2}{s}ds}(1+(\varphi(t))^2)dt+\varphi(r_0)\right],\\
&=\int_{r_0}^r e^{-(n-1)\int_t^r\fracsm{1+\varphi^2}{s}ds}(1+\varphi^2)dt+e^{-(n-1)\int_{r_0}^r\fracsm{1+\varphi^2}{s}ds}\varphi(r_0).
\end{split}
\eeq

As a simple first consequence of an inherent monotonicity, we see that (since $1=t/t\leq r/t$ for $t\in [r_0,r]$):
\beq
\begin{split}
\varphi(r) &\leq e^{-(n-1)\int_{r_0}^r\fracsm{1+(\varphi(s))^2}{s}ds}\left[r\int_{r_0}^r e^{(n-1)\int_{r_0}^t\fracsm{1+(\varphi(s))^2}{s}ds}\frac{1+(\varphi(t))^2}{t}dt+\varphi(r_0)\right]\\
& = \frac{r}{n-1}+\left[\varphi(r_0)-\frac{r}{n-1}\right]e^{-(n-1)\int_{r_0}^r\fracsm{1+(\varphi(s))^2}{s}ds}\\
& = \frac{r}{n-1} - \frac{r_0\varphi'(r_0)}{n-1}e^{-(n-1)\int_{r_0}^r\fracsm{1+(\varphi(s))^2}{s}ds}.
\end{split}
\eeq

This, together with (\ref{phiODE}), gives when $r_0>0$ that under the condition $\varphi'(r_0)\geq 0$, which the natural assumption $\varphi(r_0)\leq 0$ for example guarantees, the global bounds:
\begin{align}
&\varphi(r_0)\leq\varphi(r)\leq\frac{r}{n-1}, \quad\mathrm{for}\quad r\in [r_0,\infty),\label{phi_up_bound}\\
&\varphi'(r)\geq 0, \quad\mathrm{for}\quad r\in [r_0,\infty)
\end{align}

(Also, dividing by $1+\varphi^2$ in (\ref{phiODE}) and integrating:
\[
\frac{\pi}{2}\geq \arctan\varphi (r) = \int_r^R\left[1-\fracsm{n-1}{r}\varphi\right].
\]
Hence, since $1-\fracsm{n-1}{r}\varphi>0$, we must on a subsequence have $\varphi(r_n)/r_n\to 1$.)

We now consider a given $r_0=R_*$, so that $\varphi(R_*)=0$ (and so (\ref{phi_up_bound}) holds).

By partial integration, we get:
\begin{align}\label{partial_first}
\begin{split}
\varphi(r) &= \frac{1}{n-1}\left[r-R_*e^{-(n-1)\int_{R_*}^r\fracsm{1+(\varphi(s))^2}{s}ds}\right]\\
&\quad - \frac{1}{n-1}\int_{R_*}^r e^{-(n-1)\int_{t}^r\fracsm{1+(\varphi(s))^2}{s}ds}dt.
\end{split}
\end{align}

We apply Cauchy-Schwarz to see:
\begin{align*}
\int_{R_*}^r e^{(n-1)\int_{R_*}^t\fracsm{1+(\varphi(s))^2}{s}ds}dt&=\int_{R_*}^r e^{(n-1)\int_{R_*}^t\fracsm{1+(\varphi(s))^2}{s}ds}\sqrt{\frac{1+\varphi^2}{t}}\sqrt{\frac{t}{1+\varphi^2}}dt\\
&\leq \left(\int_{R_*}^r e^{2(n-1)\int_{R_*}^t\fracsm{1+(\varphi(s))^2}{s}ds}\frac{1+\varphi^2}{t}dt\right)^{\fracsm{1}{2}}\left(\int_{R_*}^r\frac{t}{1+\varphi^2}\right)^{\fracsm{1}{2}}\\
&= \frac{1}{\sqrt{2(n-1)}}\left(\int_{R_*}^r\frac{t}{1+\varphi^2}\right)^{\fracsm{1}{2}}\left[e^{2(n-1)\int_{R_*}^r\fracsm{1+(\varphi(s))^2}{s}ds}-1\right]^{\fracsm{1}{2}}\\
&\leq \frac{1}{\sqrt{2(n-1)}}\left(\int_{R_*}^r\frac{t}{1+\varphi^2}\right)^{\fracsm{1}{2}}e^{(n-1)\int_{R_*}^r\fracsm{1+(\varphi(s))^2}{s}ds}.
\end{align*}

All in all:
\beq\label{CauchyS}
\varphi(r)\geq \frac{r-R_*}{n-1}- \sqrt{\frac{n-1}{2}}\left(\int_{R_*}^r\frac{t}{1+\varphi^2}\right)^{\fracsm{1}{2}}.
\eeq

Thus, already by $\varphi^2\geq 0$ we generate an improved, non-trivial lower bound:
\beq\label{first_low_bound}
\varphi(r)\geq \frac{r-R_*}{n-1}-\sqrt{\frac{r^2-R_*^2}{2(n-1)}}\geq \alpha_n\frac{r-R_*}{n-1}\geq \frac{r-R_*}{4(n-1)},\qquad \alpha_n=1-\fracsm{1}{\sqrt{2(n-1)}},
\eeq

for $n\geq 2$, or when $n\geq 3$:
\[
\varphi(r)\geq\frac{r-R_*}{2(n-1)}.
\]

By integration, this implies
\[
V(r)\geq \frac{\alpha_n}{2(n-1)}(r-R_*)^2,\qquad r\geq R_*.
\]
Using (\ref{first_low_bound}) in (\ref{CauchyS}) yields, with $\beta_n=\alpha_n/(n-1)$:
\begin{align*}
\int_{R_*}^r\frac{t}{1+\varphi^2}&\leq\frac{\log\left(1+\beta_n^2(r-R_*)^2\right)}{2\beta_n^2}+\frac{R_*}{\beta_n}\arctan(\beta_n(r-R_*)),
\end{align*}

so that
\beq
\varphi(r)\geq \frac{r-R_*}{n-1}- \sqrt{\frac{n-1}{2}}\left(\frac{\log\left(1+\beta_n^2(r-R_*)^2\right)}{2\beta_n^2}+\frac{\pi R_*}{2\beta_n}\right)^{1/2}.
\eeq

Another consequence of (\ref{first_low_bound}) is:
\begin{align*}
\int_{R_*}^r e^{-(n-1)\int_t^r\fracsm{1+(\varphi(s))^2}{s}ds}dt&\leq \int_{R_*}^r e^{-\frac{\alpha_n^2}{n-1}\left[\frac{r^2-t^2}{2}-2R_*(r-t)+R_*\log(\frac{r}{t})\right]}\left(\fracsm{t}{r}\right)^{n-1} dt\\
&\leq 3R_* + \int_{4R_*}^r e^{-\frac{\alpha_n^2}{n-1}(r-t)\left[\frac{r+t}{2}-2R_*\right]}dt\\
&\leq 3R_* + e^{-\frac{\alpha_n^2}{2(n-1)}r^2}\int_{4R_*}^r e^{\frac{\alpha_n^2}{2(n-1)}rt}dt\\
&\leq 3R_* + \left(1-e^{-\frac{\alpha_n^2}{2(n-1)}(r^2 - 4R_*^2)}\right)\frac{1}{r}.
\end{align*}
where we note that while the splitting in the intermediate step is only valid when $r\geq 4R_*$, the contribution is otherwise negative so the final quantity is always an upper bound. Using this estimate with (\ref{partial_first}) gives another lower bound:
\beq\label{phi_refined}
\varphi(r)\geq \frac{r-R_*}{n-1}-\frac{2}{\alpha_n^2}\frac{1}{r}-\frac{3R_*}{n-1}\geq \frac{r-4R_*}{n-1}-\frac{24}{r},
\eeq
where various improvements, but which we will not directly need, could be made (f.ex. the $1/r$ term can be excluded when $r\leq 4R_*$).

Also, one could include the final integration in the above (and split at $t= 4R_*$), to similarly see
\begin{align*}
\int_{R_*}^r\int_{R_*}^u e^{-(n-1)\int_t^u\fracsm{1+(\varphi(s))^2}{s}ds}dtdu&\leq \int_{R_*}^r\int_{R_*}^u e^{-\frac{\alpha_n^2}{n-1}\left[\frac{u^2-t^2}{2}-2R_*(u-t)+R_*\log(\frac{u}{t})\right]}\left(\fracsm{t}{u}\right)^{n-1} dtdu\\
&\leq \fracsm{9}{2}R_*^2 + \int_{R_*}^r\int_{4R_*}^u e^{-\frac{\alpha_n^2}{n-1}(u-t)\left[\frac{u+t}{2}-2R_*\right]}dtdu\\
&\leq \fracsm{9}{2}R_*^2 + \int_{R_*}^r\int_{4R_*}^u e^{\frac{\alpha_n^2}{2(n-1)}(ut-u^2)}dtdu\\
&\leq \fracsm{9}{2}R_*^2 + \frac{2(n-1)}{\alpha_n^2}\log\left(\fracsm{r}{R_*}\right).
\end{align*}

Thus one obtains another global bound for $V$ (where now the leading coefficient is sharp):
\beq
\begin{split}
V(r) &\geq \frac{(r-R_*)^2}{2(n-1)}-\frac{2}{\alpha_n^2}\log\left(\fracsm{r}{R_*}\right)-\frac{9R_*^2}{2(n-1)}\\
&\geq \frac{(r-R_*)^2}{2(n-1)}-24\log\left(\fracsm{r}{R_*}\right)-\frac{9R_*^2}{2(n-1)}.
\end{split}
\eeq

Also, by (\ref{first_low_bound})
\[
\sup_{r\in [R_*,\infty)}\frac{r}{1+\varphi^2}\leq \sup_{r\in [R_*,\infty)}\frac{r}{1+\alpha_n^2\frac{(r-R_*)^2}{(n-1)^2}}=\frac{\sqrt{\fracsm{(n-1)^2}{\alpha_n^2}+R_*^2}}{1+\fracsm{\alpha_n^2}{(n-1)^2}\left[\sqrt{\fracsm{(n-1)^2}{\alpha_n^2}+R_*^2}-R_*\right]^2}\leq R_* + 1,
\]
for all $n\geq 2$ and $R_*\geq 0$, a fact which we will not directly need, but it is interesting to note how the global bound is insensitive to the precise value of $\alpha_n$.

One could proceed to iterate the monotonicity formula in this manner, to obtain the sharp coefficients on the higher order terms, by doing further partial integration in (\ref{partial_first})
\begin{align*}
\int_{R_*}^r e^{-(n-1)\int_{t}^r\fracsm{1+(\varphi(s))^2}{s}ds}dt&=\frac{1}{n-1}\left[\frac{r}{1+\varphi^2}-R_*e^{-(n-1)\int_{R_*}^r\fracsm{1+(\varphi(s))^2}{s}ds}\right]\\
&-\frac{1}{n-1}\int_{R_*}^r e^{-(n-1)\int_{t}^r\fracsm{1+(\varphi(s))^2}{s}ds}\frac{1-2\varphi t(1-\fracsm{n-1}{t}\varphi)}{1+\varphi^2}dt,
\end{align*}
and estimating the terms using the previous bounds. For such an estimate, in the scope of the current analysis, however, it is more convenient to simply view the bounds generated above by the monotonicity formula as a recipe for writing down explicit subsolutions which are matched at the initial condition as well as asymptotically along the end (see Section \ref{sec_subsol} below).

For the upper branch of the winglike solution $\mathcal{W}_{R}$, we now consider the equation for $r(V)$,
\beq
\frac{r''}{1+(r')^2}+r'=\frac{n-1}{r},
\eeq
or equivalently:

\beq
r''+(1+(r')^2)r'=(n-1)\frac{(1+(r')^2)}{r}.
\eeq

Freezing both nonlinearities, we view it as the equation $r''+ h(x)r'=k(x)$, so that one particular solution to the homogeneous solution is $r_1(x)=1$, and another one $r_2(x)=\int_0^x e^{-\int_0^t (1+(r'(s))^2)ds}dt$. The general solution being
\[
r(x)= C_1 + C_2r_2(x) + (n-1)\int_0^x e^{-\int_0^t (1 + (r')^2)ds}\left[\int_0^t e^{\int_0^u (1+(r')^2)ds}\fracsm{(1+(r')^2)}{r}du\right] dt,
\]
we obtain for $r(0)=R$ and $r'(0)=0$ the formula:
\beq
r(x)= R + (n-1)\int_0^x\int_0^t e^{-\int_u^t (1+(r')^2)ds}\fracsm{(1+(r')^2)}{r}\:du\:dt,
\eeq

\beq\label{rprimeq}
r'(x) = (n-1)\int_0^x e^{-\int_u^x (1+(r')^2)ds}\fracsm{(1+(r')^2)}{r}\:du
\eeq

This firstly means $r'(x)\geq0$ for $x\in[0,\infty)$, so that one may further estimate using $1/r(u)\geq 1/r(x)$, when $x\geq u$, to get:
\beq
r'(x)\geq \frac{n-1}{r(x)}\left[1-e^{-\int_0^x (1+(r')^2)ds}\right]\geq\frac{n-1}{r(x)}\left(1-e^{-x}\right).
\eeq

Integrating the inequality $(r^2)'\geq 2(n-1)(1-e^{-x})$, we get
\[
r(x)\geq \sqrt{2(n-1)(x-1+e^{-x})+R^2},
\]
so that in consequence we have the bounds on the upper branch
\beq\label{up_funnel}
V(r)\leq \frac{r^2-R^2}{2(n-1)} + 1,
\eeq
which are the bounds built into the funnel condition, Definition \ref{def_funnel}.

\subsection{Subsolutions versus dimension ($n\leq 4$ and $n>4$)}\label{sec_subsol}
We let
\[
\Psi(f)=f''-(1+(f')^2)\left[1-(n-1)\fracsm{f'}{r}\right].
\]
By a subsolution, we mean as usual a function $f$ such that $\Psi(f)\leq 0$. Consider

\[
\tau_{R_*}(r)=\frac{(r-R_*)^2}{2(n-1)}-\frac{1}{2}\log\left(1+(r-R_*)^2\right).
\]
The following lemma concerns when $\tau_{R_*}(r)$ is a universal subsolution family, i.e. for any $R_*$ with initial conditions $\psi(R_*)=0$, $\psi'(R_*)=0$.

\begin{lemma}\label{sub_sols}
For $n\geq 5$ and any $R_*>0$, the $R_*$-based asymptotics function $\tau_{R_*}$ is a subsolution to Equation (\ref{phiODE}) on $[R_*,\infty)$.
\end{lemma}
\begin{proof}
After rewriting, we see that $\Psi(f)\leq 0$ is equivalent to the inequality
\beq
P_{n,R_*}(r) \leq 0,\quad r\in [R_*,\infty),
\eeq
where $P_{n,R_*}$ denotes the polynomial

\beq
P_{n,R_*}(r) = \sum_{k=0}^{8}c_k(n,R_*)r^k,
\eeq

\begin{align*}
c_0(n,R_*)& = R_*^9 + (7 - 5 n + 
    n^2) R_*^7 + (16 - 21 n + 9 n^2 - n^3) R_*^5  \\
    &\quad + (13 - 24 n + 15 n^2 - 
    3 n^3) R_*^3 + (2 - 5 n + 4 n^2 - n^3) R_*, \\
c_1(n,R_*)& = 8 R_*^8 + (42 - 30 n + 6 n^2) R_*^6 + (67 - 
    92 n + 42 n^2 - 5 n^3) R_*^4 \\
    &\quad + (29 - 59 n + 41 n^2 - 9 n^3) R_*^2 -1 + 2 n^2 - n^3,
 \\
c_2(n,R_*)& = 28 R_*^7 + (105 - 75 n + 15 n^2) R_*^5\\
&\quad + (108 - 158 n + 78 n^2 - 10 n^3) R_*^3 + (19 - 46 n + 37 n^2 - 9 n^3) R_*,\\
c_3(n,R_*)& = 56 R_*^6 + 20(7 - 5n + n^2) R_*^4\\
& \quad + (82 - 132 n + 72 n^2 - 10 n^3) R_*^2 + 3- 11 n + 11 n^2 - 3 n^3, \\
c_4(n,R_*)& = 70 R_*^5 + 15(7 - 5 n + n^2) R_*^3 + (28 - 53 n + 33 n^2 - 5 n^3) R_*\\
c_5(n,R_*)& = 56R_*^2 + 6(7 - 5n + n^2)R_*^2 + 3 - 8 n + 6 n^2 - n^3,\\
c_6(n,R_*)& = (7 - 5n + n^2 + 28R_*^2)R_*,\\
c_7(n,R_*)& = 8R_*^2,\\
c_8(n,R_*)& = -R_*.
\end{align*}

The coefficients here will generally have mixed signs (even when $n$ is large), reflecting the fact that $P_{n,R_*}(r)$ generally has roots in $[0,\infty)$, and that on this interval $\tau_{R_*}$ is generally not a subsolution. However, changing the base point of the polynomial, to make it centered at $R_*$, we have

\beq
P_{n,R_*}(r) = \sum_{k=0}^{8}d_k(n,R_*)(r-R_*)^k,
\eeq

with the new coefficients

\begin{align*}
d_0(n,R_*)& = (-2 n^2 + 5n - 3) R_*,\\
d_1(n,R_*)& = -n^3 + 2 n^2 - 1,\\
d_2(n,R_*)& = (- 4 n^2 + 13 n - 10) R_*,\\
d_3(n,R_*)& = - 3 n^3 + 11 n^2 - 11 n + 3,\\
d_4(n,R_*)& = (- 3 n^2 + 13n -13) R_*,\\
d_5(n,R_*)& = - n^3 + 6 n^2 - 8 n + 3,\\
d_6(n,R_*)& = (- n^2 + 5 n -7) R_*,\\
d_7(n,R_*)&= 0,\\
d_8(n,R_*)& = - R_*.
\end{align*}

We see immediately that when $n\geq 5$, then for $R_*\geq 0$ all the coefficients $d_k(n,R_*)$ are non-negative. Thus, all the $P_{n,R_*}$ have no zeroes on the end $[R_*,\infty)$.

\end{proof}

Interestingly, for the subcritical dimensions $n=2,3,4$, it is generally not true that the natural quantity $\Psi(\tau_{R_*})$ has a definite sign in $[R_*,\infty)$. However, a simple additional assumption still guarantees the absence of zeros, namely:

\begin{lemma}
$\tau_{R_*}$ is a subsolution on $[R_*,\infty)$ whenever $R_*\geq 2$.
\end{lemma}

This last lemma is also easily verified via the polynomial with coefficients $d_k$, given in the proof of Lemma \ref{sub_sols}.

\begin{proof}[Proof of Lemma \ref{in_funnel}]
It is easily seen from the differential equation, that for $\varphi(R)=-\infty$, then there is some $R<R_*<\infty$, such that $\varphi(R_*)=0$. Then integrating
\[
\left(\arctan \varphi\right)'=\frac{\varphi'}{1+\varphi^2} = 1-\frac{n-1}{r}\varphi,
\]
with $\varphi(R)=-\infty$ and $\varphi(R_*)=0$, we see that since $\varphi<0$ on $[R,R_*]$:
\[
\frac{\pi}{2}=\arctan \varphi(R_*) - \arctan \varphi(R) = \int_R^{R_*}\left[1-\frac{n-1}{r}\varphi\right]\geq R-R_*.
\]
In other words,
\begin{align*}
& R_*\leq R+\fracsm{\pi}{2}.
\end{align*}

To estimate the height at the point $R_*$, look at the equation for graphs over the $V$-axis and let for convenience $\eta(\xi)=r(-\xi)$:
\[
\frac{\eta''}{1+(\eta')^2}-\eta'=\frac{n-1}{\eta}.
\]
Then $\eta(0)=R$, $\eta'(0)=0$, $\eta'>0$, and the value $d = -V(R_*)$, where $\eta(d)=R_*$, is characterized by $\eta'(d)=+\infty$. Thus again
\[
\frac{\pi}{2}=\arctan \eta'(d) - \arctan \eta'(0) \geq \int_0^{d}\frac{n-1}{\eta(\xi)} d\xi\geq \frac{(n-1)d}{R_*}.
\]

In other words,
\begin{align*}
& d\leq \frac{\pi}{2}\frac{R_*}{n-1}\leq \frac{\pi}{2}\frac{R+\fracsm{\pi}{2}}{n-1}.
\end{align*}

Combining this with Equation (\ref{up_funnel}) and the previous Lemma (in the version that $\tau_{R_*}$ is a subsolution on $[R_*,\infty)$ whenever $R_*\geq 2$), finishes the proof that the solution is contained in the solid funnel $\mathcal{F}_{R}$ in the statement of Lemma \ref{in_funnel}.

\end{proof}

Before finishing the proof of Theorem \ref{NoFunnels}, we remind the reader of the strong maximum principle, or "touching" principle. We include for completeness the short proof (which follows Corollary 1.18 in \cite{CM} closely).

\begin{lemma}[Strong maximum principle for translating solitons]\label{MP}
Let $\Sigma_1, \Sigma_2\subseteq \Reals^{n+1}$ be two smooth, connected, embedded, self-translating $n$-dimensional hypersurfaces, with translation direction $\mathbf{a}\in\Reals^{n+1}$,
\beq
H_{\Sigma_j}=\langle \mathbf{a},\nu_{\Sigma_j}\rangle,\quad |\mathbf{a}|=1,\quad j=1,2.
\eeq
Suppose $x_0\in\Sigma_1\cap\Sigma_2 \neq \emptyset$, and that within an open $\Reals^{n+1}$-ball $B_{\delta}(x_0)$, the surface $\Sigma_1$ lies entirely on one side of $\Sigma_2$, i.e. $\Sigma_1\cap B_{\delta}(x_0)$ is contained in the closure of either one of (when $\delta>0$ is possibly chosen smaller) the two connected components of $B_{\delta}(x_0)\setminus\Sigma_2$.

Then the surfaces coincide: $\Sigma_1=\Sigma_2$.
\end{lemma}
\begin{proof}
Let $x_0\in\Sigma_1\cap\Sigma_2$ be a point of contact. By possibly changing the orientations, we may assume that also the normals coincide, and we denote $\mathbf{\nu}_0:=\mathbf{\nu}_{\Sigma_1}(x_0)=\mathbf{\nu}_{\Sigma_2}(x_0)$.

After a rotation $\mathcal{R}$ such that $\mathcal{R} \nu_0 =e_{n+1}$, we may write the surfaces $\Sigma_j$, $k=1,2$, as graphs $(p,u_j(p))$ over the $(x_1,x_2,\ldots,x_n,0)$-hyperplane, by which the translation direction rotates to some possibly different unit vector $\boldsymbol\mu=\mathcal{R}\mathbf{a}$. We write $x_0=(p_0, u(p_0))$.

The equations satisfied by the $u_j$ then take the form:
\beq
\diverg\left( \frac{\nabla u_j}{\sqrt{1+|\nabla u_j|^2}}\right)=\frac{\boldsymbol\mu\cdot\begin{pmatrix}
1\\
-\nabla u_j
\end{pmatrix}
}{\sqrt{1+|\nabla u_j|^2}},\quad k=1,2.
\eeq

Subtracting one equation from the other and rewriting, we get that in terms of $v(p)=u_2(p)-u_1(p)$,
\beq\label{Equ_MP}
0=\diverg\Bigg(\sum_{i,j=1}^n (a_{ij}\partial_j v)\partial_i\Bigg) + \sum_{j=1}^n b_j \partial_j v,
\eeq
for coefficients $a_{ij}=a_{ij}(p)$ and $b_j=b_j(p)$ depending on $\nabla u_1(p)$ and $\nabla u_2(p)$.

Note that compared to the classical minimal surface case (which corresponds to $\boldsymbol\mu =0$), the only additional terms are of the form
\begin{align*}
&\frac{\boldsymbol\mu\cdot\begin{pmatrix}
1\\
-\nabla u_1
\end{pmatrix}
}{\sqrt{1+|\nabla u_1|^2}}-\frac{\boldsymbol\mu\cdot\begin{pmatrix}
1\\
-\nabla u_2
\end{pmatrix}
}{\sqrt{1+|\nabla u_2|^2}}=\frac{\boldsymbol\mu \cdot
\begin{pmatrix}
\nabla (u_2-u_1)\cdot\nabla (u_1 + u_2)\\
\nabla (u_2-u_1)
\end{pmatrix}}{\sqrt{1+|\nabla u_1|^2}\sqrt{1+|\nabla u_2|^2}}\\
&\qquad\qquad\qquad - (\boldsymbol\mu\cdot \nabla u_2)\frac{\nabla (u_2-u_1)\cdot\nabla (u_1 + u_2)}{\sqrt{1+|\nabla u_1|^2}\sqrt{1+|\nabla u_2|^2}(\sqrt{1+|\nabla u_1|^2}+\sqrt{1+|\nabla u_2|^2})},
\end{align*}
so that these additional terms do not add any "constants" (i.e. no 0-jet terms), but contribute solely to the $b_i$-coefficients in (\ref{Equ_MP}).

Now, since it was arranged that $\nabla u_1(p_0)=\nabla u_2(p_0)=0$, then also on some small ball around $p_0$, it holds that $|\nabla u_1|$ and $|\nabla u_2|$ are sufficiently small so that
\[
a_{ij}(p) \xi_i \xi_j \geq \lambda |\xi|^2,
\]
for some uniform local ellipticity constant $\lambda >0$.

Finally, by the usual strong maximum principle for (\ref{Equ_MP}), we thus conclude that $\Sigma_1$ cannot be entirely on one side of $\Sigma_2$ (within $B_\delta(x_0)$, for some small $\delta>0$) unless $\Sigma_1\cap B_\delta(x_0) = \Sigma_2\cap B_\delta(x_0)$. Then by weak unique continuation, which follows for example from the real analyticity of the $\Sigma_j$, and using connectedness, we conclude that $\Sigma_1=\Sigma_2$.
\end{proof}

We also rely essentially on the following compactness lemma.
\begin{lemma}\label{compactness}
The closure of the intersection of the complement of the funnel $\mathcal{F}^\lambda_{R_0}=\mathcal{F}^\lambda_{R_0,0}$ with the set traced out by all the winglike solutions of smaller $R$-values,
\beq
\overline{\bigcup_{0\leq R\leq R_0}\mathcal{W}_R\setminus \mathcal{F}^\lambda_{R_0}}
\eeq
is a compact subset of $\Reals^{n+1}$.
\end{lemma}
\begin{proof}
Firstly, Lemma \ref{in_funnel} reduces the statement to showing that the right-hand-side in
\[
\bigcup_{0\leq R\leq R_0}\mathcal{W}_R\setminus \mathcal{F}^\lambda_{R_0}\subseteq \bigcup_{0\leq R\leq R_0}\mathcal{F}^\lambda_R\setminus \mathcal{F}^\lambda_{R_0}
\]
is a bounded subset. This follows directly by definition of the funnel condition (note the importance of the relative pointwise directions of motion of the $f_{\pm}$ as the $R$-value decreases) in Definition \ref{def_funnel}.
\end{proof}

\begin{proof}[Proof of Theorem \ref{NoFunnels}]

Let $\Sigma^n\subseteq\Reals^{n+1}$ be a smooth, complete, $e_{n+1}$-self-translating $n$-dimensional hypersurface, which avoids the parabolic funnel $\mathcal{F}^\lambda_{R_0,y_0}$ of aperture $R_0\geq0$ and logarithmic growth $\lambda>1$. Without loss of generality, we may assume that $\Sigma$ is connected. By a parallel translation, under which (\ref{TransEq}) is naturally invariant, we may normalize the situation further by assuming that $y_0=(0,\ldots,0)\in\Reals^{n+1}$, i.e. we assume for definiteness that the funnel is positioned on the "principal" axis, and $\mathcal{F}^\lambda_{R_0}=\mathcal{F}^\lambda_{R_0,0}$.

We consider separately the cases corresponding to whether or not every member of the Clutterbuck-Schn\"u{}rer-Schulze soliton family of smaller aperture (and hence the Altschuler-Wu soliton) contain $\Sigma$ on one side or not, and if so, on which side. Namely, consider the family $\mathcal{W}_R$ for $R\in [0,R_0]$, where $R=0$ corresponds to the Altschuler-Wu soliton \cite{AW94}. We denote by $u_0(p)$ the graph generating by revolution the one-ended Altschuler-Wu soliton.

Assume first that (we will notice that for this case that $\lambda = 1$ is admissible),
\beq
\bigcap_{R\in [0,R_0]}\Sigma \cap \mathcal{W}_R \neq \emptyset.
\eeq

By Lemma \ref{in_funnel} and the exterior funnel assumption, we have
\[
\Sigma \cap \mathcal{W}_{R_0}\subseteq \Sigma \cap \mathcal{F}^\lambda_{R_0}= \emptyset,
\]
so the supremum of all values with non-trivial intersection is well-defined:
\beq\label{R_tilde_def}
\tilde{R}=\sup\big\{R\leq R_0:\Sigma\cap \mathcal{W}_R\neq\emptyset\big\}.
\eeq

To see that the value $\tilde{R}$ is assumed, consider a sequence $R_i\to \tilde{R}$ as $i\to\infty$, and corresponding points $x_i\in \Sigma\cap \mathcal{W}_{R_i}$. Since by the assumptions
\beq
x_i\in\Sigma\cap \mathcal{W}_{R_i}\:\subseteq\:\Sigma\cap\bigcup_{R\leq R_0}\mathcal{W}_{R}\:\subseteq\: \bigcup_{R\leq R_0}\mathcal{W}_{R}\setminus\mathcal{F}^\lambda_{R_0},
\eeq
the compactness in Lemma \ref{compactness} and properness of the embedding $\Sigma\hookrightarrow \Reals^{n+1}$ ensures convergence of a subsequence. That is, there exists $\tilde{x}\in\Reals^{n+1}$ such that $x_i\to \tilde{x}$ as $i\to\infty$, and $\tilde{x}\in\Sigma$. By joint continuity of the family $R\mapsto \mathcal{W}_R$ of winglike solutions in its parameters (by the standard lemma ensuring local smooth dependency on initial conditions, in the theory of ordinary differential equations), also $\tilde{x}\in \mathcal{W}_{\tilde{R}}$. Thus all in all $\tilde{x}\in\Sigma\cap \mathcal{W}_{\tilde{R}}$.

By the definition of $\tilde{R}$ in (\ref{R_tilde_def}), one sees that for some sufficiently small $\Reals^{n+1}$-ball $B_\delta(\tilde{x})$, it holds that $\mathcal{W}_{\tilde{R}}\cap B_{\delta}(\tilde{x})$ lies entirely within one connected component of $B_{\delta}(\tilde{x})\setminus\Sigma$. Namely, first choose $\delta>0$ small enough that the intersection of $B_{\delta}(\tilde{x})$ with $\Sigma$ has the topology of an $\Reals^n$-ball, as does every $\mathcal{W}_{R}\cap B_{\delta}(\tilde{x})$ for values of $R$ near $\tilde{R}$. Thus, in particular, $B_{\delta}(\tilde{x})\setminus\Sigma$ has two components. Suppose now that both of these components contain points from $\mathcal{W}_{\tilde{R}}\cap B_{\delta}(\tilde{x})$. Then, by perturbing slightly (using again continuous dependence of the family on $R$) the value $R$, one sees that $\tilde{R}$ being the supremum value in (\ref{R_tilde_def}) is violated. Thus, all conditions in the maximum principle Lemma \ref{R_tilde_def} being satisfied, we see that $\Sigma = \mathcal{W}_{\tilde{R}}$, which contradicts Lemma \ref{in_funnel}.

Assume next that
\beq\label{trace_empty}
\bigcap_{R\in [0,R_0]}\Sigma \cap \mathcal{W}_R = \emptyset,
\eeq
and furthermore for all $p=(x_1,x_2,\ldots x_n)$, the only points for which $(p,x_{n+1})\in\Sigma$ occurs have $x_{n+1}<u_0(p)$, i.e. all such points are positioned strictly lower w.r.t. the $e_{n+1}$-direction compared to the corresponding value on $\mathcal{W}_0$. In that case, consider the parallel translates
\[
\mathcal{W}^s_0= \mathcal{W}_0 - se_{n+1},
\]
where $s\in [0,\infty)$. Since, $\mathcal{W}_0$ being an entire graph, this gives a foliation of the region $\{x_{n+1}\leq u_0(p)\}$, there must exist some value $s_0$ such that
\beq
\Sigma\cap\mathcal{W}^{s_0}_0\neq \emptyset.
\eeq
Then let
\beq
\tilde{s} = \inf \{s: \Sigma\cap\mathcal{W}^s_0\neq \emptyset\}
\eeq

In this case, the asymptotics in (\ref{CSS_asympt}), again suffice to conclude the property of compactness of the additional set which is traced out when the parameter is changed, i.e. that
\beq
\Sigma\cap\bigcup_{0\leq s\leq s_0}\mathcal{W}^s_0
\eeq
is a compact set, provided that one f.ex. takes $\lambda >1$ as in our definition. The same contradiction as before is reached using Lemma \ref{MP} (here one uses once again $\lambda >1$, which in particular gives that $\mathcal{W}_0$ does not admit a funnel in its exterior).

Finally, in the case that (\ref{trace_empty}) holds while for $p=(x_1,x_2,\ldots x_n)$, the only points for which $(p,x_{n+1})\in\Sigma$ occurs have $x_{n+1}>u_0(p)$, and again (\ref{CSS_asympt}) ensures compactness of the intersection with the translates of the $\mathcal{W}_0$ soliton in the $+e_{n+1}$-direction, and with the funnel condition Definition \ref{def_funnel}, the contradiction is again reached (with $\lambda=1$ also being permissible in this last case).
\end{proof}

\bibliographystyle{amsalpha}

\end{document}